\def\ismimseprint{1}        
\ifx\ismimseprint\undefined
\documentclass[1p,preprint,dvipsnames]{elsarticle}
\else
\documentclass[11pt,dvipsnames]{scrartcl}
\fi
\usepackage[american]{babel}
\usepackage[utf8]{inputenc}

\usepackage{hyperref}
\hypersetup{breaklinks=true,colorlinks,
  urlcolor=BrickRed,citecolor=Green,linkcolor=Blue}

\usepackage{amsmath,amssymb,amsthm,amsfonts}
\usepackage{mathtools,mathabx}
\usepackage{subdepth}

\newtheorem{theorem}{Theorem}[section]
\newtheorem{lemma}[theorem]{Lemma}
\numberwithin{equation}{section}
\theoremstyle{remark}
\newtheorem{remark}[theorem]{Remark}

\usepackage{xargs} 

\usepackage{rotfig}
\usepackage{rotfig_tikz}

\usepackage{subfig}
\usepackage{pgfplots}
\usepackage{pgfplotstable}
\usepackage{xcolor}
\colorlet{colorhist}{CornflowerBlue!40!White}
\colorlet{colorhistborder}{CornflowerBlue!80!White}
\colorlet{colordist}{Black}
\usetikzlibrary{external}
\tikzsetnextfilename{note}

\newcommand{\sect}{\textcolor{black}{Section}}
\newcommand{\refchap}{\textcolor{black}{Chapter}}
\newcommand{\refsect}{\textcolor{black}{Section}}
\newcommand{\refthm}{\textcolor{black}{Theorem}}
\newcommand{\numset}[1]{\ensuremath{\mathbb{#1}}}
\newcommand{\nbyn}{\ensuremath{n \times n}}

\newcommand{\nbym}{\ensuremath{n \times m}}

\newcommand{\F}{\ensuremath{\numset{F}}}
\newcommand{\N}{\ensuremath{\numset{N}}}

\renewcommand{\Re}{\ensuremath{\mathrm{Re}}}
\renewcommand{\Im}{\ensuremath{\mathrm{Im}}}
\newcommand{\R}{\ensuremath{\numset{R}}}

\newcommand{\Rnn}{\ensuremath{\R^{\nbyn}}}

\newcommand{\C}{\ensuremath{\numset{C}}}
\newcommand{\Cn}{\ensuremath{\C^n}}
\newcommand{\Cnm}{\ensuremath{\C^{\nbym}}}
\newcommand{\Cnn}{\ensuremath{\C^{\nbyn}}}

\newcommand{\trans}[1]{\ensuremath{#1^T}}
\newcommand{\tconj}[1]{\ensuremath{#1^*}}
\newcommand{\conj}[1]{\ensuremath{\overline{#1}}}

\newcommand{\On}[1][n]{\ensuremath{O(#1)}}
\newcommand{\Oplus}[1][n]{\ensuremath{SO(#1)}}
\newcommand{\Ominus}[1][n]{\ensuremath{O^-(#1)}}
\newcommand{\Un}[1][n]{\ensuremath{U(#1)}}
\newcommand{\Uplus}[1][n]{\ensuremath{SU(#1)}}

\newcommand{\bigO}{\ensuremath{\mathcal O}}

\newcommand{\normdist}{\ensuremath{\mathcal N_{\R}}}
\newcommand{\cnormdist}{\ensuremath{\mathcal N_{\C}}}
\newcommand{\normdistf}{\ensuremath{\mathcal N_{\F}}}

\newcommand{\matnormdist}[2]{\ensuremath{\mathcal N_{\R}^{(#1,#2)}}}
\newcommand{\matcnormdist}[2]{\ensuremath{\mathcal N_{\C}^{(#1,#2)}}}
\newcommand{\matnormdistf}[2]{\ensuremath{\mathcal N_{\F}^{(#1,#2)}}}

\newcommand{\chisquared}{\ensuremath{\chi_{\R}^2}}
\newcommand{\cchisquared}[1][n-1]{\ensuremath{\chi_{\C}^2(#1)}}
\newcommand{\follows}{\ensuremath{\thicksim}} \newcommand{\wh}{\widehat}
\newcommand{\wt}{\widetilde} \newcommand{\eu}{\ensuremath{\mathrm{e}}}
\newcommand{\iu}{\ensuremath{\mathrm{i}}}

\DeclarePairedDelimiter{\norm}{\lVert}{\rVert}
\DeclarePairedDelimiter{\bignorm}{\big\lVert}{\big\rVert}
\DeclarePairedDelimiter{\Bignorm}{\Big\lVert}{\Big\rVert}
\DeclarePairedDelimiter{\abs}{\lvert}{\rvert}
\DeclareMathOperator{\diag}{diag}

\DeclareMathOperator{\Arg}{Arg}

\usepackage{tcolorbox}
\tcbuselibrary{skins}
\tcbuselibrary{breakable}
\tcbset{shield externalize}

\usepackage[algosection,vlined,algo2e,linesnumbered,resetcount,algoruled]
{algorithm2e}
\DontPrintSemicolon
\SetKwProg{Function}{function}{}{}
\IncMargin{1em}

\newcommand{\algto}{\ensuremath{\mathbf{\ to\ }}}
\newcommand{\detsign}{\ensuremath{\xi}}
\newcommand{\funcomment}[1]{%
  \begin{tcolorbox}[nobeforeafter, width=0.88\textwidth, %
    boxsep=0pt, left=4pt, right=4pt, top=4pt, bottom=4pt, %
    colback=gray!15!white, sharp corners, boxrule=0pt, frame hidden, %
    enhanced]
    \normalfont{#1}
  \end{tcolorbox}}

\newcommand{\thetitle}{Sampling the eigenvalues of random\\
  orthogonal and unitary matrices}
\newcommand{\thetitlenote}{Version of \today.}

\newcommand{\theauthori}{Massimiliano Fasi}
\newcommand{\theaddressi}{%
  School of Science and Technology, %
  \"Orebro University, %
  \"Orebro, Sweden}
\newcommand{\theemaili}{massimiliano.fasi@oru.se}
\newcommand{\thefundingi}[1][this author]{%
  The work of #1 was supported by the Royal Society, the Wenner-Gren
  Foundations grant UPD2019-0067, and the Istituto Nazionale di Alta
  Matematica INdAM-GNCS Project 2019.}

\newcommand{\theauthorii}{Leonardo Robol}
\newcommand{\theaddressii}{%
  Dipartimento di Matematica, %
  Universit{\`a} di Pisa, Italy}
\newcommand{\theemailii}{leonardo.robol@unipi.it}
\newcommand{\thefundingii}[1][This author]{%
  #1 is a member of the research group GNCS, and his work has been
  partially supported by a GNCS/INdAM project ``Giovani Ricercatori'' 2018.}

\newcommand{\theabstract}{We develop an efficient algorithm for sampling the
  eigenvalues of random 
  matrices distributed according to the Haar measure over the orthogonal or
  unitary group. Our technique samples directly a factorization of the
  Hessenberg form of such matrices, and then computes their eigenvalues with a
  tailored core-chasing algorithm. This approach requires a number of
  floating-point operations that is quadratic in the order of the matrix being
  sampled, and can be adapted to other matrix groups. In particular, we explain
  how it can be used to sample the Haar measure over the special orthogonal and
  unitary groups and the conditional probability distribution obtained by
  requiring the determinant of the sampled matrix be a given complex number on
  the complex unit circle.}

\newcommand{\thekeywords}{%
  Random matrix\sep %
  unitary matrix\sep %
  orthogonal matrix\sep %
  eigenvalue sampling\sep %
  Haar distribution}

\newcommand{\theclass}{%
  15B10\sep 
  15B52\sep 
  65F15
  }

\begin{document}

\ifx\ismimseprint\undefined
\begin{frontmatter}
  \title{\thetitle\tnoteref{t1}}
  \tnotetext[t1]{\thefundingi[the first author]
    \thefundingii[The second author]}

  \author[1]{\theauthori}
  \ead{\theemaili}
  \address[1]{\theaddressi}

  \author[2]{\theauthorii}
  \ead{\theemailii}
  \address[2]{\theaddressii}

  \begin{abstract}
    \theabstract{}
  \end{abstract}

\end{frontmatter}
\else
\newcommand{\sep}{, }
\title{\thetitle\footnote{\thetitlenote}}
\author{\theauthori\thanks{\theaddressi{} \mbox{(\theemaili)}. \thefundingi}
  \and
  \theauthorii\thanks{\theaddressii{} \mbox{(\theemailii)}. \thefundingii}}
\date{}
\maketitle
\begin{abstract}
  \noindent\textbf{Abstract.} \theabstract{}
  \bigskip


  \bigskip

  \noindent\textbf{AMS subject classifications.} \theclass.
\end{abstract}
\fi


\section{Introduction}%
\label{sec:introduction}
 
Random matrix theory, introduced by Wishart~\cite{wish28} about 90 years ago,
investigates the properties of matrices whose entries are random variables. The
quantities of interest range from the joint probability distribution of the
matrix elements to the asymptotic behaviour of its eigenvalues and singular
values~\cite{meht04}, and applications stretch from nuclear
physics~\cite{jphysa36,wign51}, wireless networks~\cite{code11}, and
neuroscience~\cite{raab06} to numerical analysis~\cite{edra05,nego47}
and number theory~\cite{mesn05}. Random matrix theory is still a very active
area of research~\cite{jphysa52}. We refer the
interested reader to the survey by Edelman and Rao~\cite{edra05} for a general
introduction to the topic, and to the
monographs by Forrester~\cite{forrester2010log} and Mehta~\cite{meht04}
for a more complete discussion.
The general
mechanisms by which random matrix theory can be employed to solve practical
problems are discussed by Edelman and Wang~\cite{edwa13}.

In applications, one is often interested in random matrices with a given
structure. In quantum mechanics, for example, the energy levels of a system are
described by the eigenvalue of its Hamiltonian, a Hermitian operator on an
infinite-dimensional Hilbert space. By approximating this space by a Hilbert
space of finite dimension, one can reduce the problem of finding the energy
levels to that of solving a Hermitian eigenvalue problem. The true Hamiltonian,
however, is typically not known, thus it is customary to make statistical
assumptions on the distribution of its entries, enforcing only the symmetry of
the operator. The distribution of the eigenvalues of random symmetric and
Hermitian matrices has been extensively studied~\cite{fuko81,meht04,trot84} and
an algorithm for sampling the eigenvalues of uniformly distributed Hermitian
matrices has been developed by Edelman, Sutton, and Wang~\cite{esw14}.

Here we are interested in
the orthogonal group
${\On = \{Q \in \Rnn \mid \trans{Q}Q = I_n\}}$ and in the unitary group
${\Un = \{U\in \Cnn \mid \tconj{U}U = I_n\}}$, where $I_n$ denotes the identity
matrix of order $n$ and \trans{A} and \tconj{A} denote the transpose and the
conjugate transpose of $A$, respectively. It is easy to prove that the
determinant of an orthogonal or unitary matrix lies on the unit circle,
and that the special orthogonal group
${\Oplus = \{Q \in \On \mid \det Q = 1\}}$ and the special unitary group
${\Uplus = \{U \in \Un \mid \det U = 1\}}$ are subgroups of $\On$ and
$\Un$, respectively. \On{} is made of two
connected components, the already mentioned
\Oplus{} and one in which all matrices have determinant $-1$, which we denote by
$\Ominus$. Clearly, the latter is not a group.

Random
unitary matrices find application in quantum physics where they are employed,
for example, to model scattering matrices and Floquet operators \cite[\refsect~2.1]{forrester2010log}. Random orthogonal matrices, on the other hand, are used
in statistical mechanics to characterize the behavior of certain log-gas systems
\cite[\refsect~2.9]{forrester2010log}.

For a group $G$, the measure $\mu$
such that $\mu(G) = 1$ is a normalized left or right Haar measure if for any
$Q \in G$ and any measurable $\mathcal G \subset G$ it satisfies
$\mu(Q\mathcal G) = \mu(\mathcal G)$ or $\mu(\mathcal GQ) = \mu(\mathcal G)$,
respectively. For compact Lie groups,
it can be shown that the left and right measures are unique and
coincide. Hence, since \On, \Oplus, \Un, and \Uplus{} are all compact Lie
groups~\cite[\refchap~1]{hall15}, they have a unique normalized (left and right)
Haar measure~\cite[\S~58, \S~60]{halm50}.

We consider the problem of sampling  efficiently the joint eigenvalue distribution of
unitary (or orthogonal) matrices distributed according to the Haar measure. Numerically,
this may be obtained by sampling matrices from
the desired group uniformly, and then computing their
eigenvalues by relying, for instance, on the QR iteration.
The latter step requires $\bigO(n^3)$ floating-point operations (flops) to
sample the $n$ eigenvalues of a matrix of order $n$. The
key observation is that for this task it is not necessary to explicitly sample
matrices from the corresponding group, but it suffices to understand the
distribution of their Hessenberg forms, which we analyze in detail in
\sect~\ref{sec:hessenberg-haar}. The main advantage of this approach is that
unitary or orthogonal matrices in Hessenberg form can be diagonalized in
$\bigO(n^2)$ flops. We will exploit this to derive the
algorithm discussed in \sect~\ref{sec:comp-eigen-unit}, which has quadratic
complexity and linear storage requirements.

The algorithm we propose can efficiently sample the joint
distribution of the eigenvalues of Haar-distributed matrices from any of the Lie
groups $\On$, $\Oplus$, $\Un$, and $\Uplus$. In
\sect~\ref{sec:experimental-results} we show that the empirical phase and
spacing of eigenvalues sampled by our algorithm follow the corresponding
theoretical distributions for $\Un$, and then we explore empirically the
distribution of the eigenvalues of matrices from the Haar distribution of
$\Uplus$, $\On$, $\Oplus$, and $\Ominus$,
for which fewer theoretical results are available.

Our starting point is an algorithm proposed by Stewart~\cite{stew80} for
sampling random matrices from the Haar distribution of \On. We recall this
approach and the subsequent generalization to \Un{}, due to Diaconis and
Shahshahani~\cite{dish87}, in \sect~\ref{sec:haar-stewart}. This technique
exploits an algorithm for the QR factorization based on Householder
transformations, which we revise in \sect~\ref{sec:householder-qr}.

These techniques require the sampling
of $\bigO(n^2)$ random
variables, and need $\bigO(n^2)$
memory for storing the result.
We provide an alternative
and more compact formulation
for the Hessenberg form obtained
by the algorithms above, which
requires the sampling of $\bigO(n)$ random
variables and $\bigO(n)$ storage. We
show that this formulation can be
used to compute the eigenvalues
in $\mathcal O(n^2)$ floating
point operations by leveraging
the unitary QR algorithm in \cite{amvw15}.

The use of a condensed factorization
for storing random matrices has been
already explored, for instance, by Edelman and Ure~\cite{edelmanreport},
who sample unitary matrices by taking
random Schur parameters \cite{grag86}.
Methods similar to the technique presented
in this work might be obtained by representing the Hessenberg forms of unitary Haar-distributed matrices
using Schur parameters, or similarly condensed forms such as CMV matrices
\cite{cmv03}, and then
using a quadratic method to compute
their eigenvalues  \cite{aurentz2016roots,amrv18,aurentz2018fast,bevilacqua2015cmv,bevilacqua2015compression,gemignani2017fast,grag86}.

The approach discussed here is based on the
unitary QR algorithm in \cite[\refsect~5]{amvw15},
The latter can be seen as a special case of the rootfinding
algorithm of Aurentz et al.~\cite{aurentz2016roots}, which has
been proven to be backward stable~\cite{aurentz2018fast},
and compares favorably with the methods above
in terms of performance~\cite{aurentz2016roots}.

Finally, we introduce some notation.
Throughout the manuscript, we use capital
letters ($A$, $B$, \dots) to denote matrices,
lower case letters ($u$, $v$, \dots) to denote vectors, and lower
case Greek letters ($\alpha$, $\beta$, \dots) to denote scalars.
We indicate the entries of matrices and vectors using a subscript
notation, so that $a_{ij}$ denotes the entry in position
$(i,j)$ of the matrix $A$ and $v_k$ refers to the $k$th element of
the vector $v$.
We use the same notation for random variables.

We denote by
$\normdist(\mu, \sigma^2)$ the Gaussian distribution centered at $\mu \in \R$
with variance $\sigma^2$,
and by $\matnormdist{m}{n}$ the distribution of $m \times n$ random matrices
with independently distributed Gaussian entries, that is,
\begin{equation*}
  X \follows \matnormdist{m}{n} \iff x_{ij} \follows \normdist(0,1),\qquad
  i = 1, \ldots, m,\quad j = 1, \ldots n.
\end{equation*}
The chi-squared distribution with $k$ degrees of freedom, denoted by
$\chisquared(k)$, is the distribution of
the sum of the squares of $k$ independent Gaussian random variables,
and is formally defined by
\begin{equation*}
  \gamma \follows \chisquared(k) \iff
  \gamma = \sum_{i=1}^k\delta_i^2,\qquad \delta_i \follows \matnormdist{k}{1}.
\end{equation*}

\newcommand{\realz}{\ensuremath{\gamma}}
\newcommand{\imagz}{\ensuremath{\delta}}
\newcommand{\independent}{\ensuremath{\perp\!\!\!\perp}}
These are real-valued distributions. The complex counterpart of
$\normdist(\mu, \sigma^2)$ is denoted by $\cnormdist(\mu, \sigma^2)$ and
defined by
\[ \realz + \iu \imagz \follows \cnormdist(\mu, \sigma^2) \iff
  \realz \follows \normdist\bigg(\Re(\mu), \frac{\sigma^{2}}{2}\bigg)
  \text{ and }
  \imagz \follows \normdist\bigg(\Im(\mu), \frac{\sigma^{2}}{2}\bigg),\quad
  \text{ with } \realz \independent \imagz,
\]
where the notation $\realz \independent \imagz$ indicates that the random
variables $\realz$ and $\imagz$ are independent.
The distribution $\matcnormdist{m}{n}$ is defined as in the real case, and we
can define the complex analogue of the $\chisquared(k)$ distribution as
\[
  \gamma \follows \cchisquared(k) \iff
  \gamma = \sum_{i = 1}^k |\delta_i|^2, \qquad
  \delta_i \follows \cnormdist^{(k, 1)}.
\]
It is easy to prove that $\cchisquared(k) \follows \chisquared(2k) / 2$.


\section{Householder transformations and QR factorization}%
\label{sec:householder-qr}

In this section we briefly recall some basic facts about Householder
transformations, and discuss how they can be employed to compute the QR
factorization of a square matrix.

Let $v \in \Cn$ be a nonzero vector. The matrix
\begin{equation}
  \label{eq:householder}
  P(v) = I_n - \frac{2}{\norm{v}_{2}^2}v\tconj{v}
\end{equation}
is a Householder transformation. It is easy to verify that $P(v)$ is unitary and
Hermitian, and in particular is orthogonal and symmetric if the entries of $v$
are real. This implies that $P{(v)}^2 = I_n$. Moreover, computing the
action of $P(v)$ on a vector requires only $\bigO(n)$ flops, instead of the
$\bigO(n^2)$ flops that would be needed for a generic $n \times n$ matrix-vector
product.

Householder transformations are a convenient tool to zero out the trailing
entries of a nonzero vector $u \in \Cn$. For instance, let
$\theta_1 = \Arg u_1$, where $\Arg \colon \C \to (-\pi, \pi]$ denotes the principal value
of the argument function. Then the matrix $P(v)$ for
$v = u + \eu^{\iu \theta_1}\norm{u}_{2}e_1$ is such that
$P(v)u = -\eu^{\iu \theta_1}\norm{u}_{2}e_1$, where $e_i$ denotes the $i$th column of
the identity matrix. In order to zero out only the last $n-k$ components of
$u \in \Cnn$, it suffices to consider the block matrix
\begin{equation}
  \label{eq:Pk}
  \wh P_k(u) :=
  \begin{bmatrix}
    I_{k-1} & \\
    & P(v)\\
  \end{bmatrix},\qquad
  v := u_{k:n} + \eu^{\iu \theta_k} \norm{u_{k:n}}_2e_1,\qquad
  \theta_k := \Arg u_k,\qquad
\end{equation}
where $u_{i:j} \in \C^{j-i+1}$ denotes the vectors that contains the entries of
$u$ from the $i$th to the $j$th inclusive.

Any matrix $A \in \Rnn$ has the QR factorization $A=:QR$, where $Q \in \On$ and
$R \in \Rnn$ is upper triangular~\cite[\refthm~5.2.1]{gova13}.
If $A$ is nonsingular, this factorization
is unique up to the sign of the diagonal entries of $R$. This result can be
extended to complex matrices: any
nonsingular matrix $A \in \Cnn$ has a unique QR
factorization $A =: QR$, where $Q \in \Un$ and $R \in \Cnn$ is upper triangular
with real positive entries along the diagonal~\cite[\refthm~7.2]{trba97}. More
generally, the QR factorization
of a full-rank matrix
is unique as long as the phases of
the diagonal entries of $R$ are fixed.

In many of the following proofs,
it will be useful to assume that the
matrix $A$
under consideration has full rank.
This is typically not restrictive,
since
rank-deficient matrices are a
zero-measure set in $\normdist^{(n,n)}$
and $\cnormdist^{(n,n)}$; we will
comment on this fact in further detail when needed.

We now explain how to compute efficiently the QR factorization of an
$n \times n$ complex matrix~$A$ by means of Householder
reflections~\cite[\refsect~5.2.2]{gova13}. The corresponding procedure for real
matrices can be obtained by employing real Householder reflectors.
Let the matrix $A^{(0)} := A$ be partitioned by columns as
\begin{equation*}
  A^{(0)} = \begin{bmatrix}
    A_1^{(0)} &\cdots& A_n^{(0)}
  \end{bmatrix},
\end{equation*}
and let $\wt P_1:=\wh P_1\big(A_1^{(0)}\big)$.
We obtain that
\begin{equation*}
  \wt P_1A^{(0)} =
  \begin{bmatrix}
    r_{11} & c \\
    0 & A^{(1)}
  \end{bmatrix},\qquad
  r_{11} = -\eu^{\iu \theta} \Bignorm{A_1^{(0)}},\quad
  A^{(1)}
  \in \C^{(n-1) \times (n-1)},\quad
  c \in \C^{1 \times (n-1)},
\end{equation*}
where $\theta$ denotes the complex sign of the top left element of the matrix $A^{(0)}$.
If we apply this procedure recursively to
the trailing submatrix $A^{(1)}$,
after $n-1$ steps we obtain
\begin{equation}
  \label{eq:R}
  \wt P_{n-1}\cdots \wt P_1A =: R,\qquad
  \wt P_k:=\wh P_k\big(A_1^{(k-1)}\big),\qquad
  k = 1, \dots, n-1,
\end{equation}
where the $Q:=\wt P_1\cdots \wt P_{n-1}$ is unitary and $R$ is upper triangular. This
algorithm produces the matrix $R$ and the factors $\wt P_i$ for
$i = 1, \dots, n-1$ in $4n^3/3$ flops~\cite[\refsect~5.2.2]{gova13}. Computing the
matrix $Q$ explicitly requires an additional $4n^3/3$
flops~\cite[\refsect~5.1.6]{gova13}, but by exploiting the structure one can
compute the action of $Q$ on a vector or matrix with only $n^2$ and $n^3$ flops,
respectively.

In order to normalize the factorization, note that if $D$ is the diagonal matrix
such that \mbox{$d_{ii} = -\eu^{-\iu \theta_i}$} where
$\theta_i$ is defined as in \eqref{eq:Pk}, then the matrix $DR$ has
positive real entries along the diagonal. Therefore, the  normalized
factorization with positive diagonal entries
in the upper triangular factor
is:
\begin{equation}
  \label{eq:normalizedQR}
  A = \wt Q \wt R,\qquad
  \wt R := D^{*}R,\quad
  \wt Q := \wt P_1 \dots \wt P_{n-1} D,\quad
  D := -\diag(\eu^{\iu \theta_1}, \dots, \eu^{\iu \theta_n}),
\end{equation}
where $\wt P_1, \dots, \wt P_{n-1}$ are as in~\eqref{eq:R}.

A matrix $H \in \Cnn$ is in upper Hessenberg form if $h_{ij} = 0$ when
$i > j + 1$. Any square matrix is unitarily similar to an upper Hessenberg
matrix, that is, for any $A \in \Cnn$ there exists a matrix $U_A \in \Un$ such
that $U_A^{\vphantom{*}}A \tconj{U}_A$ is an upper Hessenberg matrix.


\section{The Haar measure and Stewart's algorithm}%
\label{sec:haar-stewart}

Birkhoff and Gulati~\cite[\refthm~4]{bigu79} note that if the QR factorization
$A=:QR$ of an $n \times n$ matrix $A \follows \matnormdist{n}{n}$
is normalized so that the entries along the diagonal of $R$ are all positive,
then~$Q$ is distributed according to the Haar measure over \On{}.

This observation suggests a straightforward method for sampling Haar
distributed matrices from \On{}. One can simply generate an $n \times n$ real
matrix $A \follows  \matnormdist{n}{n}$,
compute its QR decomposition $A=:QR$, and normalize it as discussed in
\sect~\ref{sec:householder-qr}.
This procedure is easy to implement, since
efficient and numerically stable routines for computing the QR factorization are available
in most programming languages.

The computational cost of this technique can be approximately halved by
computing the QR factorization implicitly. Stewart~\cite{stew80} proposes an
algorithm that does not explicitly generate the random matrix $A$, but produces
the transpose of the matrix $Q$ in factored form as $D\wt P_1\dots \wt P_{n-1}$,
where $\wt P_k := \wh P_k\big(x^{(k)}\big)$ for some random vector
$x^{(k)} \follows \matnormdist{n}{1}$, and $D$ is an $n \times n$ diagonal sign
matrix whose entries are computed as on line~\ref{ln:diag} of
Algorithm~\ref{alg:action-Haar}.

This algorithm readily generalizes to the complex case, as suggested by Diaconis
and Shahshahani~\cite{dish87} and discussed in detail by Mezzadri~\cite{mezz07}.
In order to sample Haar-distributed random matrices from \Un{}, it suffices to
generate vectors with entries drawn from the standard complex normal
distribution $\cnormdist(0,1)$, and replace the real sign function by its
complex generalization $e^{\iu \Arg(z)}$ for $z \in \C$.

We outline this approach in Algorithm~\ref{alg:action-Haar}. The function
\textsc{Umult(X, \F)} computes the action of a Haar-distributed matrix from $\On$
(if $\F = \R$) or $\Un$ (if $\F = \C$) on the rectangular matrix ${X \in \Cnm}$.
In order
to determine the computational cost of the algorithm, note that asymptotically
only the two matrix-vector products and the matrix
sum on line~\ref{ln:action-Haar-cost} are significant.
Therefore, each iteration of the for loop starting on line~\ref{ln:for} requires
$4km$ flops, and the computation cost of Algorithm~\ref{alg:action-Haar}
is approximately $2n^2m$ flops.

In order to sample Haar-distributed matrices, it suffices to set $X$ to $I_n$.
In this case, the computational cost of the algorithm can be reduced by taking
into account the special structure of $A$: the cost of
line~\ref{ln:action-Haar-cost} drops to $4k^2$ flops per step, which yields an
overall computational cost of $4n^3/3$ flops.

\begin{algorithm2e}[t]
  \caption{Action of a matrix from the Haar distribution.}\label{alg:action-Haar}
  \Function{$\textsc{Umult}(X \in \Cnm, \F \in \{\R,\C\})$
    \funcomment{Compute the matrix $QX$, where $Q$ is an orthogonal (if
      $\F = \R$) or unitary (if $\F = \C$) matrix from the Haar distribution.}}{
    \For{$k \gets 2 \algto n$}{\label{ln:for}
      $v \follows \matnormdistf{k}{1}$\; 
      $d_{n-k+1} \gets -\eu^{\iu \Arg v_1}$\;\label{ln:diag}
      $u \gets v - d_{n-k+1} \norm{v}_2e_1$\;
      $u \gets u / \norm{u}_2$\;
      $X =: \begin{bmatrix}
        X_1 \\ X_2
      \end{bmatrix}
      \begin{matrix*}[l]
        \} n-k \\ \} k
      \end{matrix*}$\;
      $X \gets
      \begin{bmatrix}
        X_1 \\ X_2 - (2u)(\tconj{u}X_2)
      \end{bmatrix}$\; \label{ln:action-Haar-cost}
    }
    $z \follows \normdistf(0,1)$\;
    $\Return \diag(d_1, \dots, d_{n-1}, -\eu^{\iu \Arg z}) \cdot X$\;
  }
\end{algorithm2e}


\subsection{Sampling from the special groups}

The ideas presented so far can be modified in order to sample matrices with
prescribed determinant. Imposing that the determinant be $1$ is of particular
interest, as it implies sampling from the compact Lie groups $\Uplus$ and
$\Oplus$. As discussed in the previous section, the QR factorization of a random
matrix $A$ can be used to sample matrices distributed according to the Haar
measure over $\Un$ and $\On$. An
analogous result holds for the special groups, if the
last diagonal entry of $R$ is chosen so that $\det Q = 1$.

\begin{lemma} \label{lem:SU}
  Let $A \follows \matcnormdist{n}{n}$ \textup{(}resp. $A \follows \matnormdist{n}{n}$\textup{)}
  and let $A =: QR$ be its QR factorization, where $Q$ and $R$ are chosen so that
  \[
    r_{ii} \in \{\gamma \in \R : \gamma \ge 0\},\qquad
    r_{nn} = \det(A) \cdot \left[ \prod_{j=1}^{n-1} r_{jj}
    \right]^{-1},
  \]
  whenever $A$ is nonsingular.
  Then, $Q$ is
  distributed according to the Haar measure over $\Uplus$ \textup{(}resp. $\Oplus$\textup{)}.
\end{lemma}

\begin{proof}
	We consider the complex case first.
	Note that the set of rank deficient
	matrices has measure zero in
	$\matcnormdist{m}{n}$; therefore, the distribution
	of the unitary QR factor of such matrices is irrelevant
	for the distribution under consideration.

	When $A$ is nonsingular, fixing
	the phases of the diagonal entries of $R$ makes the
	QR factorization unique. Hence, the random variables
	$q_{ij}$, for $i,j = 1, \dots, n$ are well-defined.
	In addition, the choice of the diagonal entries of $R$
	ensures that $\det R = \det A$ and thus that
	$\det Q = 1$.

	In order to prove that $Q$ is distributed according
	to the Haar measure over $\Un$, we need to show that
	it has the same distribution as $PQ$ for any constant
	matrix $P \in \Uplus$. For any such $P$,
	the matrix $PA$ has the QR factorization $PA =: (PQ)R$.

	Being independent Gaussian random variables,
	the entries of $A$ are invariant under unitary transformations, thus
	$PA$ has the same distribution as $A$.
        The triangular QR factor of $PA$ is~$R$, which
        necessarily satisfies the normalization constraints on
        the diagonal entries.
        Therefore, $PQ$ has the same distribution as $Q$.


	The proof for the real case
	is analogous and therefore omitted.
\end{proof}
The above result yields a method for sampling the Haar
distribution of the special unitary and orthogonal groups. In the
next sections, we discuss how to make this method efficient
for sampling the corresponding eigenvalue distribution.
The approach we propose can be used
for both $\Un$ and $\On$, and
$\Uplus$ and $\Oplus$.

\begin{remark} \label{rem:detxi}
	Note that the Haar distribution of $\Uplus$ coincides with the Haar distribution of $\Un$
	conditioned to the event $\det Q = 1$. This is easily verified by checking that
	the latter measure is invariant under the action of elements in $\Uplus$. This
	suggests that the above procedure can be further
	generalized and used to sample
	the probability $\mu_{\xi}$ obtained by conditioning $\mu$ with $\det Q = \xi$ for
	some $\xi \in \mathbb S^1$, where $\mathbb S^1 = \{\xi \in \C : \abs{\xi} = 1\}$
	denotes the complex unit circle.
	If $\xi \neq 1$, these matrices do not form a group, but we
	can write
	\[
	  \{ Q \in U(n) \ | \ \det Q = \xi \} =
	  \{ P\wt Q \ | \ \wt Q \in SU(n) \},
	\]
	where $P$ is any constant matrix such that $\det P = \xi$. Sampling the matrices in
	$\Uplus$ and then multiplying them by any fixed $P$ yields the correct
	conditional probability
	distribution. More specifically,
	by choosing the diagonal matrix $P = \diag(1, \ldots, 1, \xi)$
	we can readily adapt the algorithm discussed in the next section to
	sample unitary or orthogonal matrices with determinant $\xi$.
\end{remark}

\subsection{The eigenvalue distribution}

Given a random matrix sampled according to one of the measures described so far,
we are interested in describing the distribution of a generic eigenvalue. This
can be computed as a marginal probability by integrating the joint
eigenvalue distribution with respect to $n - 1$ variables.
For $\Un$, the latter is known explicitly \cite[\refchap~11]{meht04}, and
can be used to prove that a generic eigenvalue is uniformly distributed over
$\mathbb S^1$.

We are unaware of an analogous result for $\Uplus$, and we could not
find any references stating the expected distribution for a generic eigenvalue.
Nevertheless, using the fact that the eigenvalue distribution arises from
Haar-distributed matrices, we can obtain the partial characterization in the
following lemma.
The well-known distribution for $\Un$ is for ease of comparison.

\begin{lemma} \label{lem:su-periodic}
	Let $\mu$ and $\mu_{1}$ be the Haar distributions over $\Un$ and $\Uplus$, respectively,
	and let $\Lambda_\mu$ and $\Lambda_{\mu_1}$ be the corresponding distributions for a generic
	eigenvalue. Then, $\Lambda_{\mu}$ is the uniform distribution over $\mathbb S^1$, and
	$\Lambda_{\mu_1}$ has a  $\frac{2\pi}{n}$-periodic phase, that is,
	\begin{equation} \label{eq:xi-periodic}
	\Lambda_{\mu_1}(\mathcal G) = \Lambda_{\mu_1}\Big(e^{\frac{2\pi i}{n}}\mathcal G\Big), \qquad
	\text{for any measurable set } \mathcal G \subset \mathbb S^1.
	\end{equation}
\end{lemma}

\begin{proof}
	We start considering the distribution of $\Un$.
	Recall that $\mu$ is invariant under left multiplication
	in $U(n)$, and since $\xi I \in \Un$ for any $\xi \in \mathbb S^{1}$, we have that
	$\Lambda_{\mu}$ is invariant under multiplication by $\xi \in \mathbb S^1$. Being
	$\mathbb S^1$ a compact Lie group, $\Lambda_\mu$ must be its Haar measure, which
	is the uniform distribution.

	We can use a similar argument for $\Uplus$. Since $\xi I \in \Uplus$ for any
	$\xi$ such that $\xi^n = 1$, we have that $\Lambda_{\mu_1}$ is invariant
	under multiplication by a root of the unity, thus must be
	$\frac{2\pi}{n}$-periodic as in~\eqref{eq:xi-periodic}.
\end{proof}

In \sect~\ref{sec:experimental-results} we will verify this claim experimentally, to test the correctness
of our implementation.
In particular, we will find that $\Lambda_\mu$ is the uniform distribution, as expected, and
that $\Lambda_{\mu_1}$ has the periodicity predicted by Lemma~\ref{lem:su-periodic}.

\section{The Hessenberg form of Haar-distributed matrices}%
\label{sec:hessenberg-haar}

As mentioned in the introduction, the unitary QR algorithm of \cite{amvw15}
cannot be applied directly to the representation of the upper Hessenberg form of
a random matrix. In this section, first we show how to sample a factorization of
the upper Hessenberg form of Haar-distributed matrices using only $\bigO(n)$
random variables, then we explain how to rewrite this factorization in a form
that is suitable for computing the eigenvalues with core-chasing algorithms,
which we briefly review in \sect~\ref{sec:comp-eigen-unit}. The main result of
this section is the following.
\begin{theorem}\label{thm:2x2fact}
  Let $w_1, \ldots, w_{n-1} \in \C^2$ be independent random vectors such that
  \[
    w_j = \begin{bmatrix}
      \alpha_j \\ \beta_j
    \end{bmatrix}, \qquad
    \alpha_j \follows \cnormdist(0, 1), \qquad
    \beta_j^2 \follows \cchisquared[n - j],
  \]
  and let
  \begin{equation}
    \label{eq:reduced-householder}
    H = P_1 \ldots P_{n-1} D \in \Cnn
  \end{equation}
  be the unitary Hessenberg matrix such that
  \begin{equation}
    \label{eq:reduced-householder-pieces}
    P_j = I - \frac{2}{\norm{v_j}_2^2} v_j \tconj{v_j},\quad
    D = -\diag(\eu^{\iu \theta_1}, \dots, \eu^{\iu \theta_n}),\quad
    v_j =
    \begin{bmatrix}
      0_{j-1} \\ \alpha_j + \eu^{\iu\theta_j}\norm{w_j}_2 \\ \beta_j \\ 0_{n-j-1}
    \end{bmatrix},
  \end{equation}
  where $\theta_j = \Arg \alpha_j$ for $j = 1, \dots, n-1$ and
  $\theta_n \follows U(-\pi,\pi]$. Then, the joint eigenvalue distribution of
  $H$ is that of Haar-distributed unitary matrices of $\Un$.
\end{theorem}

\begin{proof}
In view of the discussion in \sect~\ref{sec:haar-stewart}, if $Q$ is Haar
distributed then
\begin{equation}
  \label{eq:Q-def}
  Q \follows \wt P_1 \ldots \wt P_{n-1} D,
\end{equation}
with $\wt P_j$ defined as follows:
	\[
	  \wt P_j = P(v_j) = I - \frac{2}{\norm{v_j}^2} v_j \tconj{v_j}, \qquad
	  v_j := \begin{bmatrix}
	    0_{j-1} \\
	    u^{[j]}_1 + \eu^{\iu \theta_j} \bignorm{u^{[j]}}_2 \\
	    u^{[j]}_{2:n} \\
	  \end{bmatrix}, \qquad
	  u^{[j]} \follows \cnormdist^{(n-j+1,1)},
	\]
	where $\theta_j = \Arg\big(u^{[j]}_1\big)$, and $D$ is a diagonal matrix
	defined by $d_{jj} = \eu^{\iu \theta_j}$. In order to prove the claim
	we will reduce this matrix to upper Hessenberg form.

	More specifically, we prove by induction on $n$ that there exists a
        unitary matrix $U$ such that $H \follows UQU^*$
        is upper Hessenberg, and that $Ue_1 = e_1$.
        The latter property will be useful in the induction step. Throughout the
        proof, we will also repeatedly exploit the fact that if $W$ is
        orthogonal then $WP(v)W^* = P(Wv)$, which can be verified by a direct
        computation.

	If $n = 1$, then $Q$ and $H$ are both
	diagonal matrices, and there is nothing to prove.
	If $n = 2$, then $Q$ is already upper Hessenberg, and we may write it as
	\[
	  Q = \wt P_1 D, \qquad
	  \wt P_1 = I -  \frac{2}{\norm{v_1}^2} v_1 \tconj{v_1}, \qquad
	  v_1 := \begin{bmatrix}
	    \alpha_1 + e^{i \theta_1} \bignorm{u^{[1]}} \\
	    \beta_1
	  \end{bmatrix}, \qquad
	  u^{[1]} := \begin{bmatrix}
	    \alpha_1 \\ \beta_1
	  \end{bmatrix}.
	\]
	With the matrix $U = \mathrm{diag}(1, \eu^{-\iu\theta})$, where
        $\theta := \Arg(\beta_1)$, we can perform the similarity transformation
	\[
	  U Q \tconj U = U \wt P_1 D U^* = U \wt P_1 U^* D =
	  P(U v_1) D, \qquad
	  Uv_1 = \begin{bmatrix}
	    \alpha_1 + e^{i \beta_1} \bignorm{u^{[1]}} \\
	    |\beta_1|
	  \end{bmatrix}.
	\]
	We now observe that $|\beta_1|^2 \follows \cchisquared[1]$ and
        $\bignorm{u^{[1]}}_2 = \norm*{
          \begin{bsmallmatrix}
            \alpha_1 \\
            |\beta_1| \\
          \end{bsmallmatrix}}_2$, which implies that $P(U v_1)$ and $P_1$ have
        the same distribution.
	Moreover, the matrix $U$ thus constructed satisfies $Ue_1 = e_1$.

	For the inductive step, assume that the statement holds for matrices of
        order $n-1$, and consider~$Q$ as in~\eqref{eq:Q-def}. If
        $\wt P_1 = I - \frac{2}{\norm{v_1}^2} v_1 \tconj{v_1}$, we can construct
        a Householder reflector $U_1$ such that $U_1 e_1 = e_1$, and
	\begin{equation} \label{eq:U1v1}
	  U_1 v_1 = U_1 \begin{bmatrix}
	    \alpha_1 + \eu^{\iu \theta_1} \bignorm{u^{[1]}} \\
	    u^{[1]}_{2:n}
	  \end{bmatrix} = \begin{bmatrix}
	    \alpha_1 + \eu^{\iu\theta_1} \bignorm{u^{[1]}} \\
	    \bignorm{u^{[1]}_{2:n}}_2 \\
	    0_{n-2} \\
	  \end{bmatrix}.
	\end{equation}
	We note that $U_1$ can be chosen so to be independent of
        $\wt P_2, \ldots, \wt P_{n-1}$, as it depends only on $u^{[1]}$. We can then
        consider the similarity
	\[
	  U_1 Q U_1^* = U_1 \wt P_1 \tconj{U_1} U_1 \wt P_2 \ldots \wt P_{n-1}
	    D U_1^*,
	\]
	where $U_1 \wt P_1 \tconj{U_1} = I - \frac{2}{\norm{v_1}^2} U_1 v_1 \tconj{(U_1 v_1)}
	\follows P_1$ in view \eqref{eq:U1v1}, and
	$\bignorm{u^{[1]}_{2:n}}_2^2 \follows \cchisquared$.
	We now factorize $D$ as
	\[
	  D := \begin{bmatrix}
	  1 \\
	  & d_2 \\
	  && \ddots \\
	  &&& d_n \\
	  \end{bmatrix} \begin{bmatrix}
	   d_1 \\
	   & 1 \\
	   && \ddots \\
	   &&& 1
	  \end{bmatrix} = \wh D D_1,
	\]
        and note that $U_1 D_1 = D_1 U_1$ thanks to $Ue_1 = e_1$.
        Therefore we can write
	\[
	  U_1 Q \tconj{U_1} \follows P_1 
            U_1 \wt P_2 \ldots \wt P_{n-1} \wh D \tconj{U_1}
            D_1 = P_1 \begin{bmatrix}
              1 \\
              & \wh Q
            \end{bmatrix} D_1.
          \]
	 We observe that since $U_1$ is independent of $\wt P_2 \ldots \wt P_{n-1}$
         and $\wh D$, $\wh Q$ is Haar distributed in $U(n-1)$. By inductive
         hypothesis, there exists an $(n-1) \times (n-1)$ unitary matrix $\wh U$
         which satisfies $\wh U e_1 = e_1$ and is such that
	 \[
	   \wh U \wh Q \wh U^* = \wh P_2 \ldots \wh P_{n-1} \wh D,\qquad
           P_j \follows
           \begin{bmatrix}
             1 & \\ & \wh P_j
           \end{bmatrix},\qquad
           j = 2, \dots, n-1.
	 \]
         The property $\wh Ue_1 = e_1$
	 implies that $
         \begin{bsmallmatrix}
           1 & \\ & \wh U
         \end{bsmallmatrix}
         $commutes with both $P_1$ and $D_1$, and we can write
	 \begin{align*}
           \begin{bmatrix}
             1 & \\ & \wh U
           \end{bmatrix}
           U_1Q \tconj U_1
           \begin{bmatrix}
             1 & \\ & \tconj{\wh U}
           \end{bmatrix}
           &\follows
           P_1
           \begin{bmatrix}
             1 & \\ & \wh U
           \end{bmatrix}
           \begin{bmatrix}
             1 & \\ & \wh Q
           \end{bmatrix}
           \begin{bmatrix}
             1 & \\ & \tconj{\wh U}
           \end{bmatrix}
           D_1 \\
	   &= P_1 P_2 \ldots P_{n-1} \wh D D_1 =  P_1 P_2 \ldots P_{n-1} D,
         \end{align*}
	 which concludes the proof.
\end{proof}

The analogue of Theorem~\ref{thm:2x2fact} for real matrices is obtained by
replacing the first element of $w_j$ by $\alpha_j \follows \normdist(0, 1)$ and
by sampling $\theta_n$ uniformly from the set $\{0, \pi\}$. The result can be
easily modified in order to sample the joint eigenvalues distribution of
matrices from the Haar measure over $\Uplus$ and $\Oplus$: setting
\begin{equation*}
  d_{nn} = (-1)^{n-1}\prod_{i=1}^{n-1}d_{ii}
\end{equation*}
guarantees that $\det H = 1$, while Lemma~\ref{lem:SU} ensures that the matrices
are sampled according to the Haar measure.

\begin{remark} \label{rem:su-xi}
	In a similar way, we may set the last diagonal entry of $D$
	to obtain $\det{Q} = \xi$, for any $\xi \in \mathbb S^1$.
	According to Remark~\ref{rem:detxi}, this procedure samples
	the Haar distribution conditioned on the event $\det Q = \xi$.
\end{remark}


\section{Computing the eigenvalues of
	upper Hessenberg unitary matrices}%
\label{sec:comp-eigen-unit}

By Theorem~\ref{thm:2x2fact}, any random upper Hessenberg unitary matrix $H$ can
be described by $\bigO(n)$ parameters by using the factored
form~\eqref{eq:reduced-householder}.

The computation of the eigenvalues of unitary upper Hessenberg matrices was
first considered by Gragg~\cite{grag86}, and later investigated by numerous
authors, see for instance~\cite{ammar1992implementation,gemignani2005unitary,gragg1990divide}. Here, in particular, we are interested in
the approach proposed by Aurentz, Mach, Vandebril, and Watkins~\cite{amvw15}.
This algorithm,
briefly described in \sect~\ref{sec:comp-eigen-hess}, is implemented in \texttt{eiscor}~\cite{amrv18}, a Fortran 90 package
for the solution of eigenvalue problems
by core-chasing methods available on
GitHub.\footnote{\url{https://github.com/eiscor/eiscor}} The software
computes the eigenvalues of the Hessenberg matrix
\begin{equation}
  \label{eq:Hrotations}
  H = G_1 \cdots G_{n-1} D,
\end{equation}
where the unitary matrices \mbox{$G_1$, \dots, $G_{n-1}$} are plane rotations of
the form
\begin{equation}
  \label{eq:rotation}
  G_j
  = \left[
    \begin{array}{@{}cccc@{}}
      I_{j-1} & & \\
      & \wh G_j & \\
      & & I_{n-j-1}
    \end{array}\right],
  \qquad \wh G_j =
  \begin{bmatrix}
    c_j & s_j \\
    -s_{j} & \overline{c}_j \\
  \end{bmatrix},
  \qquad c_j \in \C,
  \qquad s_j \in \R.
\end{equation}

Because of its special structure, the matrix $G_j$ in~\eqref{eq:rotation} is
said to be ``essentially $2 \times 2$'', and the $2 \times 2$ matrix
$\wh G_j$ is called a \emph{core block}. In principle, the core
chasing algorithm in~\cite{amvw15} could be applied to any factorization of $H$
involving only ``essentially $2 \times 2$'' unitary matrices, even though the
particular implementation described in \cite{amrv18} involves only the special
family of plane rotations. In practice, however, the key operation in the QR
algorithm---the so-called \emph{turnover}---has to be implemented with care in
order to ensure backward stability. In order to leverage the analysis done for
rotations of the form~\eqref{eq:rotation}, it is thus convenient to refactorize
$H$ given in the form~\eqref{eq:reduced-householder} as a product of the the
form~\eqref{eq:Hrotations}.

This section is structured as follows. First, we show how to refactorize a
representation in terms of $2 \times 2$ Householder transformations into one
consisting only of plane rotations with real sines. Then we briefly recall the
main ideas underlying the unitary QR algorithm implemented in \texttt{eiscor}.

\subsection{Refactoring core transformations}%
\label{sec:refact-core-transf}

The assumption that all core blocks in the factorization of the Hessenberg
matrix $H$ be plane rotations with real sines is not restrictive, as it is
always possible to rewrite~$H$ as a product of the form~\eqref{eq:Hrotations}.
This refactorization can be performed by noting that any $2 \times 2$ unitary
matrix $U$ can be written as
\begin{equation} \label{eq:UGD}
  U = \begin{bmatrix}
    c & s \\
    -s & \overline c
  \end{bmatrix} \begin{bmatrix}
    d_1 \\
    & d_2
  \end{bmatrix}, \qquad
  \begin{cases}
    c = u_{11} e^{\iu \theta}, \\
    s = -u_{21} e^{\iu \theta},
  \end{cases}\qquad
  \theta = \Arg(\conj{u_{21}}).
\end{equation}
where the diagonal entries are given by
\[
  d_1 = e^{-\iu \theta} |u_{11}|^2 + e^{\iu \theta} u_{21}^2, \qquad
  d_2 = e^{\iu \theta} ( u_{11} u_{22} - u_{21} u_{12} ).
\]
The procedure above can be performed in a
backward stable manner, as it coincide with
the computation of the QR decomposition of
$U$ \cite{amrv18}.

In addition, the product of a plane rotation with real sines~$G$ and a unitary
diagonal $2 \times 2$ matrix $D$ can be refactored so to swap the order of the two
operations. In fact, there exist a unitary $2 \times 2$ diagonal matrix $\wt D$ and a
plane rotation with real sines $\wt G$ such that $GD = \wt D \wt G $. This
property is easy to verify, and represents the foundation of most core-chasing
algorithms~\cite{amrv18}. Combining these two observations gives the following.

\begin{lemma}
  \label{lem:refact}
  Let $H \in \Cnn$ be a unitary upper Hessenberg matrix factored as
  in~\eqref{eq:reduced-householder}. Then, there exist $G_1$, \dots,
  $G_{n-1} \in \Cnn$ unitary plane rotations with real sines and $\wt D \in \Cnn$
  unitary diagonal such that
  \begin{equation} \label{eq:fact-rot}
    H = G_1 \ldots G_{n-1} \wt D.
  \end{equation}
  This refactorization can be computed in $\bigO(n)$ flops.
\end{lemma}

\begin{proof}
  The proof is by induction on $n$. For $n = 2$ we have that $H = P_1 D$, and
  the refactorization can be performed directly by relying on \eqref{eq:UGD}. If
  $n > 2$, then there exist a plane rotation $G_1 \in \Cnn$ as in~\eqref{eq:rotation} and
  a unitary diagonal matrix
  \begin{equation*}
    D_1 :=
    \begin{bmatrix}
      \alpha & & \\
      & \beta & \\
      & & I_{n-2}
    \end{bmatrix},\qquad
    \alpha, \beta \in \mathbb S^1,
  \end{equation*}
  such that $P_1 = G_1D_1$. Since the matrix
  $\begin{bsmallmatrix}
    d_{11} & \\
    & I_{n-1}
  \end{bsmallmatrix}$
  commutes with $D_1$, $P_2$, $\dots$, $P_{n-1}$, we can write
  \begin{equation*}
    H = G_1 
    \begin{bmatrix}
      \alpha d_{11} & \\
      & I_{n-1}
    \end{bmatrix}
    \wt P_2 P_3 \cdots P_{n-1} \wt D,\qquad
    \wt P_2 := \begin{bmatrix}
      1 & & \\
      & \beta & \\
      & & I_{n-2}
    \end{bmatrix} P_2,
  \end{equation*}
  where $\wt d_{ij} = 1$ if $i=j=1$ and $\wt d_{ij} = d_{ij}$ otherwise.
  We note that this refactorization has the
  form $H =  G_1\begin{bsmallmatrix}
      \alpha d_1 \\
      & 1
    \end{bsmallmatrix}
    \begin{bsmallmatrix}
      1 & \\
      & \wt H
    \end{bsmallmatrix}$, where $\wt H$ has the same
  structure as $H$ but order $n - 1$. The inductive hypothesis yields
  $
  \begin{bsmallmatrix}
    1 & \\
    & \wt H
  \end{bsmallmatrix} =
  G_2 \ldots G_{n-1}
  \begin{bsmallmatrix}
    1 & \\
    & D'
  \end{bsmallmatrix}$, which gives
  \[
    H = G_1 G_2 \ldots G_{n-1} \wt D,\qquad
    \wt D :=
    \begin{bmatrix}
      \alpha d_{11} \\
      & D'
    \end{bmatrix}.
  \]
  This procedure provides an algorithm for refactoring $H$ from the
  form~\eqref{eq:reduced-householder} to the form~\eqref{eq:fact-rot}.
  Noting that each step requires $\bigO(1)$ flops, for a total of
  $\bigO(n)$ flops for the complete refactorization, concludes the proof.
\end{proof}

\subsection{Computing the eigenvalues of unitary Hessenberg matrices}%
\label{sec:comp-eigen-hess}

We have described how to generate unitary upper Hessenberg matrices whose
joint eigenvalue distribution follows the Haar measure,
and we have shown how to write it in the factored
form~\eqref{eq:fact-rot}.

In order to compute the spectrum of $H$ in the form~\eqref{eq:fact-rot} in
$\bigO(n^2)$ flops, we rely on the method
proposed in \cite{amvw15}, which belongs to the
family of \emph{core-chasing algorithms}
\cite{amrv18}. Here we  provide a high-level overview of this technique, and
refer the interested reader to the original paper~\cite{amvw15} for a detailed
discussion.

With the term \emph{core transformation} we indicate an essentially $2 \times 2$
unitary matrix such as, for example, a plane rotation.  The factorization \eqref{eq:fact-rot} is
an example of a matrix expressed as product of $n - 1$ core transformation and a
diagonal matrix. In this particular case, the facotrization can also be used to give a
compact representation of $H$ that uses only $\bigO(n)$ parameters, as opposed
to the $\bigO(n^2)$ that would be necessary if all the entries of $H$ were
explicitly stored. Each core transformation acts on a pair
of adjacent indices. The matrix $G_j$ in~\eqref{eq:rotation}, for example, acts on the indices $j$ and $j+1$.

The standard single-shift bulge chasing QR algorithm works as follows. Given an
upper Hessenberg matrix $H$, we determine a first core transformation $Q_1$
acting on the indices $1$ and $2$ such that $Q_1 (H - \rho I) e_1 = \alpha e_1$.
The parameter $\rho$ is the \emph{shift}, and has to be carefully chosen in
order to ensure fast and reliable convergence of the method \cite{watkins2007matrix}. The implementation
considered here uses a Wilkinson shift, which is projected onto $\mathbb S^1$
as the matrix $H$ is unitary~\cite{amvw15}.

We use the core transformation $Q_1$ to compute $Q_1 H \tconj{Q_1}$,
which is not upper Hessenberg having a nonzero element in position $(3,1)$.
Another core transformation $Q_2$ acting on the indices $2$ and $3$ is used to
restore the upper Hessenberg structure and obtain $Q_2 Q_1 H \tconj{Q_1}$. The
similarity $Q_2 Q_1 H \tconj{Q_1} \tconj{Q_2}$, however, yields a matrix that is
not upper Hessenberg because of a nonzero element in position $(4,2)$, and the
process is repeated until the nonzero element, the so-called \emph{bulge}, is
eliminated from the last row of the matrix.
The focus on the nonzero element that breaks the upper Hessenberg
structure and is ``chased to the bottom'' until it disappears from the matrix,
justifies the name \emph{bulge-chasing QR} used for this algorithm.

Core-chasing algorithms have a similar formulation, and indeed are
mathematically equivalent~\cite{amrv18}, but do not handle the entries of the
matrix directly, as we now explain.

An upper Hessenberg matrix can always be written as $H = QR$, where $R$ is upper
triangular and $Q = G_1 \ldots G_{n-1}$ is the product of unitary plane
rotations. The core-chasing step starts by computing a core transformation $Q_1$
such that $Q_1 (H - \rho I) e_1$ is a scalar multiple of $e_1$. Keeping $H = QR$
in factored form, the similarity transformation with $Q_1$ gives
\[
  Q_1 H \tconj{Q_1} = Q_1 \underbrace{G_1 \ldots G_{n-1} R}_{H} \tconj{Q_1}.
\]
We now make the following two key observations.
\begin{itemize}
\item Given an upper triangular matrix $R$ and a core transformation $Q_1$, it
  is always possible to find another upper triangular matrix $\wt R$ and core
  transformation $\wt Q_1$ such that $R Q_1 = \wt Q_1 \wt R$. Using the
  terminology of core-chasing algorithms, the computation of $\wt Q_1$ and
  $\wt R$ from $Q_1$ and $R$ is a \emph{passthrough} operation, and can be
  represented pictorially as
  
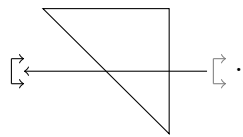
\begin{center}
    \begin{tikzpicture}[baseline={(current bounding box.center)},scale=1.66,y=-1cm]
      \uppertriangular{0.0}{0.0}{};
      \tikzrotation[gray]{1.4}{.4}{}
      \tikzrotation{-.2}{.4}{}
      \shiftthroughrl{1.4}{-.2}{0.4}
    \end{tikzpicture}\;.
  \end{center}
\item Given the matrices $G_i$ and $K_{i}$ acting on the $i$th and $i+1$st
  indices, and $J_{i+1}$ acting on the $i+1$st and the $i+2$nd indices, the
  product $G_i J_{i+1} K_i$ can be refactored as $\wt G_{i+1} \wt J_i \wt K_{i+1}$,
  where $\wt G_{i+1}$ and $\wt K_{i+1}$ act on the $i+1$st and $i+2$nd indices,
  and $\wt J_i$ acts on the $i$th and $i+1$st indices. Similarly, this operation
  can be represented pictorially as
  
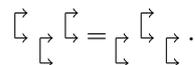
\begin{center}
    \begin{tikzpicture}[scale=1.66, y = -1cm]
      \tikzrotation{-1.0}{0.0}
      \tikzrotation{-0.8}{0.2}
      \tikzrotation{-0.6}{0.0}
      \node at (-0.4,0.2) {$=$};
      \tikzrotation{-0.2}{0.2}
      \tikzrotation{-0.0}{0.0}
      \tikzrotation{0.2}{0.2}
    \end{tikzpicture}\raisebox{0.31cm}{\;.}
  \end{center}
  In the context of core-chasing algorithms, it is more natural to reinterpret
  the refactorization above as moving the rightmost core transformation to the
  left, which can be displayed as
  \begin{center}
    \begin{tikzpicture}[scale=1.66, y = -1cm]
      \tikzrotation{-1.2}{0.2}
      \tikzrotation{-1.0}{0.0}
      \tikzrotation{-0.8}{0.2}
      \tikzrotation{-0.6}{0.0}
      \turnoverrl{-0.6}{0.0}{-1.2}{0.2}
    \end{tikzpicture}\raisebox{0.31cm}{\;.}
  \end{center}
  Clearly, all the rotations involved in the step above change, but from the
  point of view of the structure, it is as if only one rotation had moved.
  This operation is called \emph{turnover}.
\end{itemize}
With this notation, we can rephrase the factorization obtained after introducing
the core transformation $Q_1$ as
\begin{equation}
  \label{eq:qhq}
  \vcenter{\hbox{%
  \begin{tikzpicture}[scale=1.66, y=-1cm]
    \node at (0.2, .5) {$Q_1 H \tconj{Q_1} = $};
    \tikzrotation{1}{0};
    \foreach \j in {0, 0.2, ..., 0.8} {
      \tikzrotation{1.2+\j}{\j};
    }
    \uppertriangular{1.6}{0}{$R$};
    \tikzrotation{2.8}{0};
    \node at (3.2,.5) {$=$};
    \foreach \j in {0, 0.2, ..., 0.8} {
      \tikzrotation{3.6+\j}{\j};
    }
    \tikzrotation{3.4}{0.2};
    \turnoverrl{4.0}{0}{3.4}{0.2};
    \tikzrotation[gray]{4.0}{0};
    \uppertriangular{4.2}{0}{$\wt R$};
    \shiftthroughrl{5.4}{4.0}{0};
    \tikzrotation[gray]{5.4}{0};
  \end{tikzpicture}}}.
\end{equation}
In the rightmost factorization, we have first fused the top-left rotations, and
then used the \emph{passthrough} and \emph{turnover} to take the rotation that
was on the right to the left. If we call this new core transformation $Q_2$, we
can now perform the similarity transformation $Q_2 Q_1 H \tconj{Q_1} \tconj{Q_2}$ and obtain
the matrix
\begin{equation*}
    \vcenter{\hbox{%
  \begin{tikzpicture}[scale=1.66, y=-1cm]
    \node at (0.2, .5) {$Q_2 Q_1 H \tconj{Q_1} \tconj{Q_2} = $};
    \foreach \j in {0, 0.2, ..., 0.8} {
      \tikzrotation{1.2+\j}{\j};
    }
    \uppertriangular{1.6}{0}{$\wt R$};
    \tikzrotation{2.8}{0.2};
  \end{tikzpicture}}}.
\end{equation*}
The structure of this matrix is similar to that of
$Q_1 H \tconj{Q_1}$, in~\eqref{eq:qhq}, but the rightmost core transformation
has moved down one step. Indeed, it now acts on indices $2$ and $3$ instead of $1$ and $2$.
Carrying on this process will move it further down, until
it is fused at the bottom. At the very end, the core transformation will hit the
bottom rotation in the sequence $G_1 \ldots G_{n-1}$, and they will be
fused together. This completes the chasing, and is mathematically equivalent to chasing the
bulge into the bottom-right corner.

There are a few more technical details to address in order to obtain a complete
algorithm. One key point is how to detect deflations, that is, eigenvalues that
have converged. In the usual bulge-chasing setting, we monitor subdiagonal
elements, setting them to zero as soon as they become ``small enough''. For
core-chasing algorithms, we observe that a subdiagonal element is small if and
only if the corresponding rotation in the sequence $G_1 \ldots G_{n-1}$ is close
to the identity. In fact, this technique is often more accurate than
the customary criterion in practice~\cite{amrv18}.

The main computational step in this process is the refactorization
$R Q_1 = \wt Q_1 \wt R$, which requires $\bigO(n)$ flops in general. Therefore,
a single core-chasing step requires $\bigO(n^2)$ flops, and $\bigO(n^3)$ flops
will be necessary to compute the eigenvalues of a generic Hessenberg matrix, if
about $\bigO(n)$ steps are required for the QR iteration to converge. All the
other operations require only $\bigO(1)$ flops, and contribute only a low-order
term to the total cost.

However, if $H$ is unitary to begin with, then so is the upper triangular matrix
$R$. As upper triangular unitary matrices must be diagonal, the passthrough
operation can be performed in $\bigO(1)$ flops, making the cost of the QR
algorithm quadratic instead of cubic in $n$.
For a more detailed analysis, we refer the reader to the paper where the
algorithm was first introduced~\cite{amvw15}.

\subsection{Sampling the eigenvalues of random unitary and orthogonal matrices}

The approach underlying the discussion in \sect s~\ref{sec:hessenberg-haar}
and~\ref{sec:comp-eigen-unit} is summarized in Algorithm~\ref{alg:new-alg}.

The function \textsc{SampleEigs} samples the joint distribution of orthogonal or
unitary matrices from a specific distribution determined by the value of third
parameter $\xi$. If $\xi$ is 0, then the function samples the eigenvalues of Haar
distributed matrices from the orthogonal group, if $\F = \R$, or from the
unitary group, if $\F = \C$. If $\xi \neq 0$, the algorithm samples the eigenvalue
distribution of matrices whose determinant has the same phase as $\xi$. We recall
that these matrices form a group only if $\xi = 1$, in which case the algorithm
samples the eigenvalue distribution of Haar-distributed matrices from the
special orthogonal group $\Oplus$, if $\F = \R$, or special unitary group
$\Uplus$, if $\F = \C$.

In order to achieve this, we note that for $H$ in~\eqref{eq:fact-rot},
we have that $\det H = \det D = \eu^{\iu \theta}$ for some
$\theta \in (-\pi, \pi]$, since the determinant of plane rotations is 1. Therefore, once the first $n-1$ entries of $D$ are chosen, it
suffices to choose $d_n = \eu^{\iu \Arg \xi - \iu \Arg (\prod_{j=1}^{n-1}d_j)}$,
which ensures that $\det D = \eu^{\iu \Arg \xi}$.
In the pseudocode, the function \textsc{UnitaryQR} denotes a call to the
\texttt{eiscor} routine, which computes the eigenvalues of a product of plane
rotations of the form \eqref{eq:rotation}.

The computational cost of the algorithm can be determined by noticing
that each step of the for loop on line~\ref{ln:new-alg-for} requires only a
constant number of operations, which implies that the whole preprocessing taking
place between line~\ref{ln:new-alg-start} and line~\ref{ln:new-alg-end} requires
only $\bigO(n)$ flops.

\begin{algorithm2e}[t]
  \caption{Sample the joint distribution of random matrices
    matrices.}\label{alg:new-alg}
  \Function{$\textsc{SampleEigs}(n \in \N,
    \F \in \{\R,\C\},
    \detsign \in \F)$
    \funcomment{Sample eigenvalues of orthogonal (if $\F = \R$) or unitary
      (if $\F = \C$) matrices.
      If $\detsign \neq 0$, the determinant of the sampled matrices has the same
      phase as $\detsign$.
      If $\detsign = 0$, the matrices are sampled from $O(n)$ (if $\F = \R$) or
      $U(n)$ (if $\F = \C$).}}{
    $\delta \gets 1$\;\label{ln:new-alg-start}
    \For{$k \gets 1 \algto n-1$}{\label{ln:new-alg-for}
      $v_1 \follows \normdistf(0,1)$\;
      $v_2 \follows \sqrt{\chisquared(n-k)}$\;
      $d_k \gets -\eu^{\iu \Arg v_1}$\;
      $v_1 \gets v_1 - d_k \norm{v}_2$\;
      $U \gets \begin{bsmallmatrix} \delta & \\ & 1 \end{bsmallmatrix}
      \big(I - \frac{2}{\norm{v}_2}v\tconj{v}\big)$\;
      $\varphi \gets \eu^{\iu \Arg \conj{u_{21}}}$\;
      $c_k \gets \varphi\, u_{11}$\;
      $s_k \gets -\varphi\, u_{21}$\;
      $d_k \gets d_k(\conj{\varphi}\,\abs{u_{11}}^2+\varphi\,u_{21}^2)$\;
      $\delta \gets \varphi\,\det U$\;
    }
  }
  \eIf{$\detsign \neq 0$}{
    $d_n \gets \eu^{\iu \Arg \xi - \iu \Arg \big(\prod_{j=1}^{n-1}d_j\big)} $\;
  }{
    $z \gets \normdistf(0,1)$\;
    $d_n \gets -\eu^{\iu \Arg z}$\;\label{ln:new-alg-end}
  }
  \Return $\textsc{UnitaryQR}(c,s,d)$\;
\end{algorithm2e}


\newcommand{\newalg}{\texttt{sampleeig}}
\newcommand{\naivealg}{\texttt{samplemat}}
\pgfplotsset{width=0.46\textwidth}
\ifx\ismimseprint\undefined
\pgfplotsset{height=2.1in}
\else
\pgfplotsset{height=2.4in}
\fi

\section{Experimental results}%
\label{sec:experimental-results}

In this section we first validate the new algorithm experimentally, and then
compare its performance with that of the na{\"i}ve method for sampling the joint
eigenvalues distribution of Haar-distributed unitary matrices. The experiments
were run in MATLAB 9.8.0 (R2020a) Update 4 on a GNU/Linux machine equipped with
an Intel Xeon E5-2640 v3 CPU running at 2.60GHz.

In our tests we compare the following implementations.
\begin{itemize}
\item \naivealg, an algorithm that generates a Haar-distributed unitary matrix
  by calling the function $\textsc{Umult}$ in Algorithm~\ref{alg:action-Haar} on
  the identity matrix, and then computes its eigenvalues by using the built-in
  MATLAB function \verb|eig|.
\item \newalg, an implementation of Algorithm~\ref{alg:new-alg} that exploits
  the \texttt{eiscor} package to run the QR algorithm on the unitary matrix in
  factored form.
\end{itemize}

Our implementations of \naivealg\ and \newalg\ are available on Github.%
\footnote{\url{https://github.com/numpi/random-unitary-matrices}}
For reproducibility, the repository also includes the scripts we used to
run the tests reported here.


\subsection{Unitary matrices}
We start by considering the joint distribution of the eigenvalues of
Haar-distributed matrices in $\Un$ and $\Uplus$.

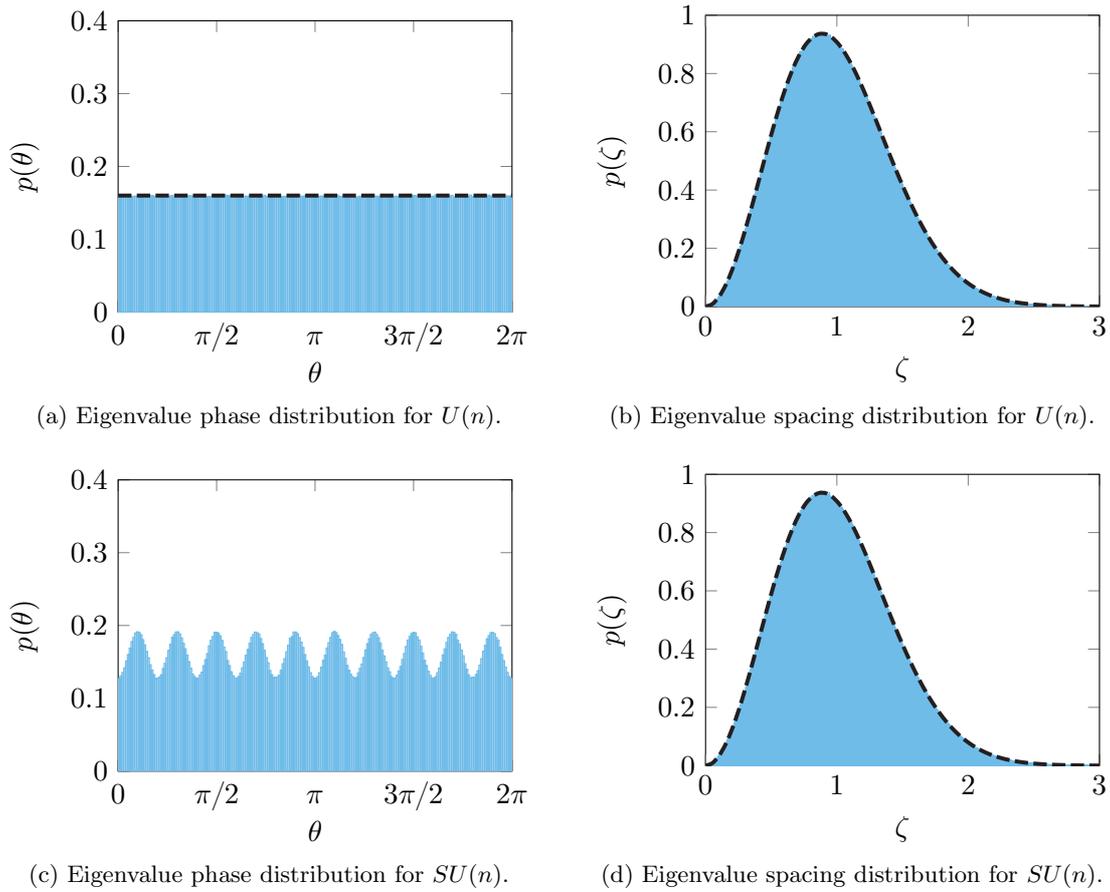
\begin{figure}[h!]
  \centering
  \subfloat[Eigenvalue phase distribution for $\Un$.%
  \label{fig:eig-dist-Un}]{%
    \begin{tikzpicture}
      \begin{axis}[
        height = .26\textheight,
        xlabel={$\theta$},
        ylabel={$p(\theta)$},
        ymin=0, ymax=0.4,
        xmin=0, xmax=6.28318530718,
        domain={0:6.28318530718},
        xtick={0, 1.57079632679, 3.14159265359, 4.71238898038, 6.28318530718},
        xticklabels={$\vphantom{/}0$,$\pi/2$, $\vphantom{/}\pi$, $3\pi/2$,
          $\vphantom{/}2\pi$},
        ]
        \pgfplotstableread{figs/eig-dist-unitary-10-0.dat}\mydata;
        \addplot [ybar interval, fill=colorhist,draw=colorhistborder]
        table [x index=0, y index=1] {\mydata};
        \addplot [dash pattern=on 5pt off 2pt,
        colordist, mark=none, line width=1.5pt] {0.16};
      \end{axis}
    \end{tikzpicture}}\hspace{15pt}
  \subfloat[Eigenvalue spacing distribution for $\Un$.%
  \label{fig:eig-spacing-Un}]{%
    \begin{tikzpicture}
      \begin{axis}[
        height = .26\textheight,
        xlabel={$\zeta$},
        ylabel={$p(\zeta)$},
        ymin=0, ymax=1,
        xmin=0, xmax=3,
        domain={0:3},
        samples=60,
        smooth,
        extra x ticks={0},
        extra x tick labels={$\vphantom{/}$},
        ]
        \pgfplotstableread{figs/eig-spacing-unitary-10-0.dat}\mydata;
        \addplot [ybar interval, fill=colorhist,draw=colorhistborder]
        table [x index=0, y index=1] {\mydata};
        \addplot [dash pattern=on 5pt off 2pt,
        colordist, mark=none, line width=1.5pt]
        {32*(x/3.14159265)^2*exp(-4*x^2/3.14159265)};
      \end{axis}
    \end{tikzpicture}}\\
  \subfloat[Eigenvalue phase distribution for $\Uplus$.%
  \label{fig:eig-dist-SUn}]{%
    \begin{tikzpicture}
      \begin{axis}[
        height = .26\textheight,
        xlabel={$\theta$},
        ylabel={$p(\theta)$},
        ymin=0, ymax=.4,
        xmin=0, xmax=6.28318530718,
        domain={0:6.28318530718},
        xtick={0, 1.57079632679, 3.14159265359, 4.71238898038, 6.28318530718},
        xticklabels={$\vphantom{/}0$,$\pi/2$, $\vphantom{/}\pi$, $3\pi/2$,
          $\vphantom{/}2\pi$},
        ]
        \pgfplotstableread{figs/eig-dist-unitary-10-1.dat}\mydata;
        \addplot [ybar interval, fill=colorhist,draw=colorhistborder]
        table [x index=0, y index=1] {\mydata};
      \end{axis}
    \end{tikzpicture}}\hspace{15pt}
  \subfloat[Eigenvalue spacing distribution for $\Uplus$.%
  \label{fig:eig-spacing-SUn}]{%
    \begin{tikzpicture}
      \begin{axis}[
        height = .26\textheight,
        xlabel={$\zeta$},
        ylabel={$p(\zeta)$},
        ymin=0, ymax=1,
        xmin=0, xmax=3,
        domain={0:3},
        samples=60,
        smooth,
        extra x ticks={0},
        extra x tick labels={$\vphantom{/}$},
        ]
        \pgfplotstableread{figs/eig-spacing-unitary-10-1.dat}\mydata;
        \addplot [ybar interval, fill=colorhist,draw=colorhistborder]
        table [x index=0, y index=1] {\mydata};
        \addplot [dash pattern=on 5pt off 2pt,
        colordist, mark=none, line width=1.5pt]
        {32*(x/3.14159265)^2*exp(-4*x^2/3.14159265)};
      \end{axis}
    \end{tikzpicture}}
  \caption{Phase (left) and spacing (right) distribution of the eigenvalues of
    1,000,000 random unitary matrix of order 10 sampled from the unitary group
    (top) and from the special unitary group (bottom) using \newalg. The dashed
    lines represent the uniform distribution over the interval $[0,2\pi)$ (left)
    and the Wigner surmise in~\eqref{eq:wig-surmise}
    (right).}\label{fig:dist-spacing-unitary}
\end{figure}

Figure~\ref{fig:dist-spacing-unitary} reports the phase distribution and the
spacing of the eigenvalues of 1,000,000 unitary matrices of order 10 sampled from
the unitary group (top row) and from the special unitary group (bottom row)
using \newalg. The histograms in the four plots are normalized so that the total
area of the columns is one. In this way, the histograms can be interpreted as
empirical probability densities and can be compared directly with the
probability density functions they are expected to match.

Let $\eu^{\iu\theta_1}$, \dots, $\eu^{\iu\theta_{n}}$ be the eigenvalue of the
matrix $A \in \Un$ normalized so that for $i$ between~1 and $n$ the phase
$\theta_i$ lies in the interval $[0, 2\pi)$. The histogram in
Figure~\ref{fig:eig-dist-Un} shows the distribution of the phases of the
10,000,000 sampled eigenvalues, whereas the dashed lines indicates the
probability density function of the uniform distribution over the interval
$[0,2\pi)$. As the eigenvalues of unitary matrices lie on the unit circle, our
results indicates that the eigenvalues sampled by the procedure are uniformly
distributed.

Next we investigate the statistical correlation among the eigenvalues sampled by
\newalg. In Figure~\ref{fig:eig-spacing-Un} we plot the probability density
function of the normalized distance between pairs of consecutive eigenvalues,
defined by
\begin{equation*}
  \zeta_i := \frac{n}{2\pi}(\theta_{i+1}-\theta_i),\qquad\theta_{n+1}
  :=\theta_1,\qquad i=1, \dots, n,
\end{equation*}
where the eigenvalues are ordered so that $\theta_1 \le \dots \le \theta_n$. In
this case the dashed line represents the theoretical spacing distribution of
Haar-distributed unitary matrices, known as Wigner
surmise~\cite[\refsect~1.5]{meht04}
\begin{equation}
  \label{eq:wig-surmise}
  p(\zeta) = \frac{\pi\zeta}{2}\exp \left( -\frac{\pi}{4}\zeta^2 \right).
\end{equation}
The empirical distribution of the sampled eigenvalues matches closely the
theoretical one, confirming that the matrices whose eigenvalues \newalg{}
samples are in fact Haar distributed.

To the best of our knowledge, the probability distribution for the phase and
spacing of Haar-distributed matrices in the special unitary group are not known in
closed form,
but we can use \newalg{} to obtain a relative frequency distribution based on
10,000,000 samples. The results in Figure~\ref{fig:eig-dist-SUn} and
Figure~\ref{fig:eig-spacing-SUn} suggest that the phase of the eigenvalues of
these matrices is not uniformly distributed, but the spacing appears to be the
same as for Haar-distributed matrices in $U(n)$, as the empirical distribution
matches the Wigner surmise in~\eqref{eq:wig-surmise}.

We note that the invariance under
multiplication by elements in $\Uplus$ implies
that $Q$ and $\diag(\xi, \dots, \xi) Q$ must have the same distribution
for any $\xi$ such that $\xi^n = 1$, and the phase of the density
of the eigenvalue distribution needs to be $2\pi / n$-periodic,
as proven in Lemma~\ref{lem:su-periodic}. This is clearly visible in
Figure~\ref{fig:eig-dist-SUn}.


\subsection{Orthogonal matrices}

\begin{figure}[p]
  \centering
  \subfloat[Eigenvalue phase distribution for \On.%
  \label{fig:eig-dist-On-even}]{%
    \begin{tikzpicture}
      \begin{axis}[
        height = .26\textheight,
        xlabel={$\theta$},
        ylabel={$p(\theta)$},
        ymin=0, ymax=3.6,
        xmin=0, xmax=6.28318530718,
        domain={0:6.28318530718},
        xtick={0, 1.57079632679, 3.14159265359, 4.71238898038, 6.28318530718},
        xticklabels={$\vphantom{/}0$,$\pi/2$, $\vphantom{/}\pi$, $3\pi/2$,
          $\vphantom{/}2\pi$},
        ytick={0,0.9,1.8,2.7,3.6}
        ]
        \pgfplotstableread{figs/eig-dist-orthog-10-0.dat}\mydata;
        \addplot [ybar interval, fill=colorhist,draw=colorhistborder]
        table [x index=0, y index=1] {\mydata};
      \end{axis}
    \end{tikzpicture}}\hspace{15pt}
  \subfloat[Eigenvalue spacing distribution for \On.%
  \label{fig:eig-spacing-On-even}]{%
    \begin{tikzpicture}
      \begin{axis}[
        height = .26\textheight,
        xlabel={$\zeta$},
        ylabel={$p(\zeta)$},
        ymin=0, ymax=1,
        xmin=0, xmax=3,
        domain={0:3},
        samples=60,
        smooth,
        extra x ticks={0},
        extra x tick labels={$\vphantom{/}$},
        ]
        \pgfplotstableread{figs/eig-spacing-orthog-10-0.dat}\mydata;
        \addplot [ybar interval, fill=colorhist,draw=colorhistborder]
        table [x index=0, y index=1] {\mydata};
        \addplot [dash pattern=on 5pt off 2pt,
        colordist, mark=none, line width=1.5pt]
        {32*(x/3.14159265)^2*exp(-4*x^2/3.14159265)};
      \end{axis}
    \end{tikzpicture}}\\
  \subfloat[Eigenvalue phase distribution for \Oplus.%
  \label{fig:eig-dist-Oplus-even}]{%
    \begin{tikzpicture}
      \begin{axis}[
        height = .26\textheight,
        xlabel={$\theta$},
        ylabel={$p(\theta)$},
        ymin=0, ymax=3.6,
        xmin=0, xmax=6.28318530718,
        domain={0:6.28318530718},
        xtick={0, 1.57079632679, 3.14159265359, 4.71238898038, 6.28318530718},
        xticklabels={$\vphantom{/}0$,$\pi/2$, $\vphantom{/}\pi$, $3\pi/2$,
          $\vphantom{/}2\pi$},
        ytick={0,0.9,1.8,2.7,3.6}
        ]
        \pgfplotstableread{figs/eig-dist-orthog-10-1.dat}\mydata;
        \addplot [ybar interval, fill=colorhist,draw=colorhistborder]
        table [x index=0, y index=1] {\mydata};
      \end{axis}
    \end{tikzpicture}}\hspace{15pt}
  \subfloat[Eigenvalue spacing distribution for \Oplus.%
  \label{fig:eig-spacing-Oplus-even}]{%
    \begin{tikzpicture}
      \begin{axis}[
        height = .26\textheight,
        xlabel={$\zeta$},
        ylabel={$p(\zeta)$},
        ymin=0, ymax=1,
        xmin=0, xmax=3,
        domain={0:3},
        samples=60,
        smooth,
        extra x ticks={0},
        extra x tick labels={$\vphantom{/}$},
        ]
        \pgfplotstableread{figs/eig-spacing-orthog-10-1.dat}\mydata;
        \addplot [ybar interval, fill=colorhist,draw=colorhistborder]
        table [x index=0, y index=1] {\mydata};
        \addplot [dash pattern=on 5pt off 2pt,
        colordist, mark=none, line width=1.5pt]
        {32*(x/3.14159265)^2*exp(-4*x^2/3.14159265)};
      \end{axis}
    \end{tikzpicture}}\\
  \subfloat[Eigenvalue phase distribution for \Ominus.%
  \label{fig:eig-dist-Ominus--even}]{%
    \begin{tikzpicture}
      \begin{axis}[
        height = .26\textheight,
        xlabel={$\theta$},
        ylabel={$p(\theta)$},
        ymin=0, ymax=3.6,
        xmin=0, xmax=6.28318530718,
        domain={0:6.28318530718},
        xtick={0, 1.57079632679, 3.14159265359, 4.71238898038, 6.28318530718},
        xticklabels={$\vphantom{/}0$,$\pi/2$, $\vphantom{/}\pi$, $3\pi/2$,
          $\vphantom{/}2\pi$},
        ytick={0,0.9,1.8,2.7,3.6}
        ]
        \pgfplotstableread{figs/eig-dist-orthog-10--1.dat}\mydata;
        \addplot [ybar interval, fill=colorhist,draw=colorhistborder]
        table [x index=0, y index=1] {\mydata};
      \end{axis}
    \end{tikzpicture}}\hspace{15pt}
  \subfloat[Eigenvalue spacing distribution for \Ominus.%
  \label{fig:eig-spacing-Ominus-even}]{%
    \begin{tikzpicture}
      \begin{axis}[
        height = .26\textheight,
        xlabel={$\zeta$},
        ylabel={$p(\zeta)$},
        ymin=0, ymax=1,
        xmin=0, xmax=3,
        domain={0:3},
        samples=60,
        smooth,
        extra x ticks={0},
        extra x tick labels={$\vphantom{/}$},
        ]
        \pgfplotstableread{figs/eig-spacing-orthog-10--1.dat}\mydata;
        \addplot [ybar interval, fill=colorhist,draw=colorhistborder]
        table [x index=0, y index=1] {\mydata};
        \addplot [dash pattern=on 5pt off 2pt,
        colordist, mark=none, line width=1.5pt]
        {32*(x/3.14159265)^2*exp(-4*x^2/3.14159265)};
      \end{axis}
    \end{tikzpicture}}
  \caption{Phase (left) and spacing (right) distribution of the eigenvalues of
    1,000,000 random orthogonal matrix of order 10 sampled from the orthogonal
    group (top), the special orthogonal group (middle), and the connected
    component of the orthogonal group that contains only matrices with negative
    determinant (bottom) using \newalg. The dashed line in the right column
    represent the Wigner surmise
    in~\eqref{eq:wig-surmise}.}\label{fig:dist-spacing-orthogonal-even}
\end{figure}
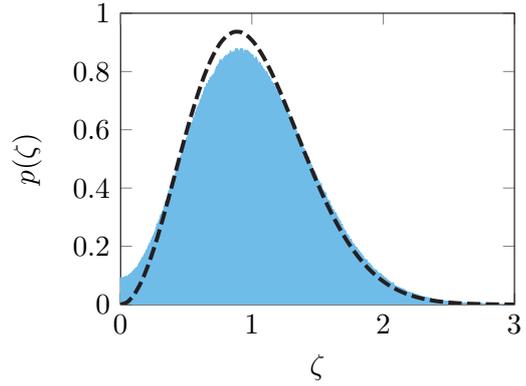
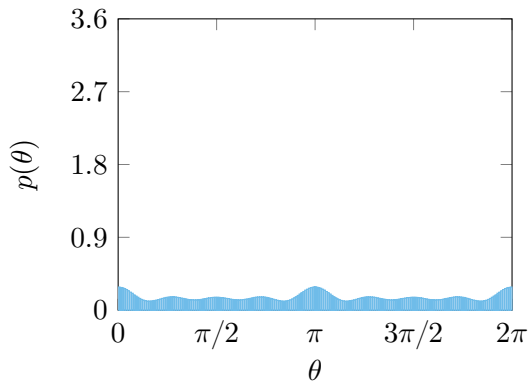
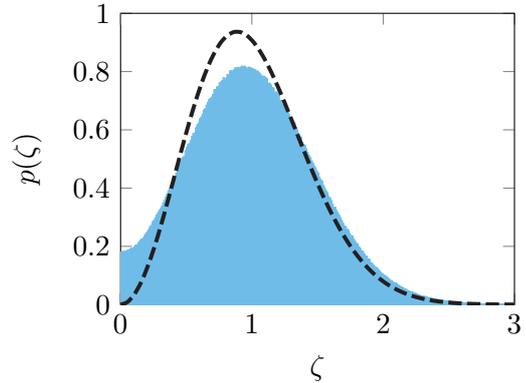
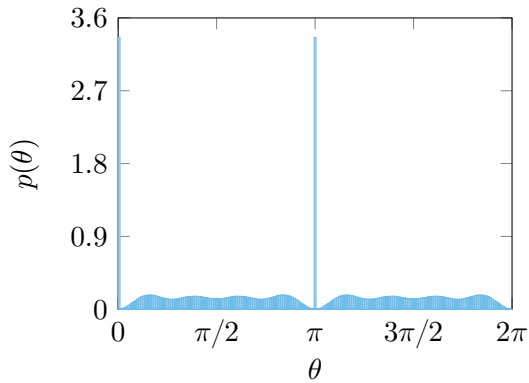
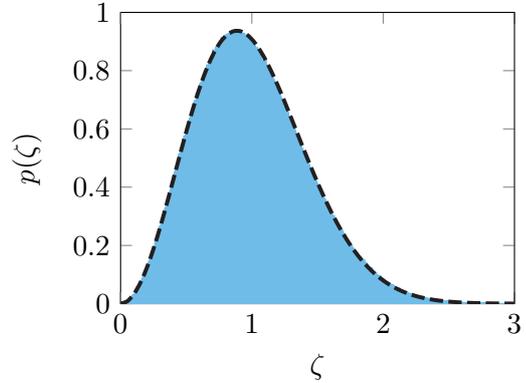

\begin{figure}[p]
  \centering
  \subfloat[Eigenvalue phase distribution for \On.%
  \label{fig:eig-dist-On-odd}]{%
    \begin{tikzpicture}
      \begin{axis}[
        height = .26\textheight,
        xlabel={$\theta$},
        ylabel={$p(\theta)$},
        ymin=0, ymax=4,
        xmin=0, xmax=6.28318530718,
        domain={0:6.28318530718},
        xtick={0, 1.57079632679, 3.14159265359, 4.71238898038, 6.28318530718},
        xticklabels={$\vphantom{/}0$,$\pi/2$, $\vphantom{/}\pi$, $3\pi/2$,
          $\vphantom{/}2\pi$},
        ytick={0,1,2.0,3.0,4.0}
        ]
        \pgfplotstableread{figs/eig-dist-orthog-09-0.dat}\mydata;
        \addplot [ybar interval, fill=colorhist,draw=colorhistborder]
        table [x index=0, y index=1] {\mydata};
      \end{axis}
    \end{tikzpicture}}\hspace{15pt}
  \subfloat[Eigenvalue spacing distribution for \On.%
  \label{fig:eig-spacing-On-odd}]{%
    \begin{tikzpicture}
      \begin{axis}[
        height = .26\textheight,
        xlabel={$\zeta$},
        ylabel={$p(\zeta)$},
        ymin=0, ymax=1,
        xmin=0, xmax=3,
        domain={0:3},
        samples=60,
        smooth,
        extra x ticks={0},
        extra x tick labels={$\vphantom{/}$},
        ]
        \pgfplotstableread{figs/eig-spacing-orthog-09-0.dat}\mydata;
        \addplot [ybar interval, fill=colorhist,draw=colorhistborder]
        table [x index=0, y index=1] {\mydata};
        \addplot [dash pattern=on 5pt off 2pt,
        colordist, mark=none, line width=1.5pt]
        {32*(x/3.14159265)^2*exp(-4*x^2/3.14159265)};
      \end{axis}
    \end{tikzpicture}}\\
  \subfloat[Eigenvalue phase distribution for \Oplus.%
  \label{fig:eig-dist-Oplus-odd}]{%
    \begin{tikzpicture}
      \begin{axis}[
        height = .26\textheight,
        xlabel={$\theta$},
        ylabel={$p(\theta)$},
        ymin=0, ymax=4,
        xmin=0, xmax=6.28318530718,
        domain={0:6.28318530718},
        xtick={0, 1.57079632679, 3.14159265359, 4.71238898038, 6.28318530718},
        xticklabels={$\vphantom{/}0$,$\pi/2$, $\vphantom{/}\pi$, $3\pi/2$,
          $\vphantom{/}2\pi$},
        ytick={0,1,2.0,3.0,4.0}
        ]
        \pgfplotstableread{figs/eig-dist-orthog-09-1.dat}\mydata;
        \addplot [ybar interval, fill=colorhist,draw=colorhistborder]
        table [x index=0, y index=1] {\mydata};
      \end{axis}
    \end{tikzpicture}}\hspace{15pt}
  \subfloat[Eigenvalue spacing distribution for \Oplus.%
  \label{fig:eig-spacing-Oplus-odd}]{%
    \begin{tikzpicture}
      \begin{axis}[
        height = .26\textheight,
        xlabel={$\zeta$},
        ylabel={$p(\zeta)$},
        ymin=0, ymax=1,
        xmin=0, xmax=3,
        domain={0:3},
        samples=60,
        smooth,
        extra x ticks={0},
        extra x tick labels={$\vphantom{/}$},
        ]
        \pgfplotstableread{figs/eig-spacing-orthog-09-1.dat}\mydata;
        \addplot [ybar interval, fill=colorhist,draw=colorhistborder]
        table [x index=0, y index=1] {\mydata};
        \addplot [dash pattern=on 5pt off 2pt,
        colordist, mark=none, line width=1.5pt]
        {32*(x/3.14159265)^2*exp(-4*x^2/3.14159265)};
      \end{axis}
    \end{tikzpicture}}\\
  \subfloat[Eigenvalue phase distribution for \Ominus.%
  \label{fig:eig-dist-Ominus-odd}]{%
    \begin{tikzpicture}
      \begin{axis}[
        height = .26\textheight,
        xlabel={$\theta$},
        ylabel={$p(\theta)$},
        ymin=0, ymax=4,
        xmin=0, xmax=6.28318530718,
        domain={0:6.28318530718},
        xtick={0, 1.57079632679, 3.14159265359, 4.71238898038, 6.28318530718},
        xticklabels={$\vphantom{/}0$,$\pi/2$, $\vphantom{/}\pi$, $3\pi/2$,
          $\vphantom{/}2\pi$},
        ytick={0,1,2.0,3.0,4.0}
        ]
        \pgfplotstableread{figs/eig-dist-orthog-09--1.dat}\mydata;
        \addplot [ybar interval, fill=colorhist,draw=colorhistborder]
        table [x index=0, y index=1] {\mydata};
      \end{axis}
    \end{tikzpicture}}\hspace{15pt}
  \subfloat[Eigenvalue spacing distribution for \Ominus.%
  \label{fig:eig-spacing-Ominus-odd}]{%
    \begin{tikzpicture}
      \begin{axis}[
        height = .26\textheight,
        xlabel={$\zeta$},
        ylabel={$p(\zeta)$},
        ymin=0, ymax=1,
        xmin=0, xmax=3,
        domain={0:3},
        samples=60,
        smooth,
        extra x ticks={0},
        extra x tick labels={$\vphantom{/}$},
        ]
        \pgfplotstableread{figs/eig-spacing-orthog-09--1.dat}\mydata;
        \addplot [ybar interval, fill=colorhist,draw=colorhistborder]
        table [x index=0, y index=1] {\mydata};
        \addplot [dash pattern=on 5pt off 2pt,
        colordist, mark=none, line width=1.5pt]
        {32*(x/3.14159265)^2*exp(-4*x^2/3.14159265)};
      \end{axis}
    \end{tikzpicture}}
  \caption{Phase (left) and spacing (right) distribution of the eigenvalues of
    1,000,000 random orthogonal matrix of order 9 sampled from the orthogonal
    group (top), the special orthogonal group (middle), and the connected
    component of the orthogonal group that contains only matrices with negative
    determinant (bottom) using \newalg. The dashed line in the right column
    represent the Wigner surmise
    in~\eqref{eq:wig-surmise}.}\label{fig:dist-spacing-orthogonal-odd}
\end{figure}

The joint eigenvalue probability density functions for
$\Oplus$ and $\Ominus$ are reported explicitly
in~\cite[\refsect~2.6]{forrester2010log} and \cite{sosh20,weyl46}.
The corresponding joint eigenvalue distribution
for the orthogonal group can be obtained easily,
since a matrix in  $\On$ belongs with equal probability
to $\Oplus$ or $\Ominus$. The eigenvalue distribution for
such matrices can be obtained integrating out $n-1$ variables
and are defined in terms of the diagonal correlation kernel of the process \cite{dish94}. However, excluding the case of unitary
matrices, such expressions are
not easy to evaluate; in most cases the limiting distribution
for large $n$ can be explicitly determined, and is typically
the uniform distribution over $\mathbb S^1$.

We can use \newalg{} to get the empirical
distribution of the phase and spacing of the eigenvalues of these matrices. In
Figure~\ref{fig:dist-spacing-orthogonal-even}, we report the relative frequency
distribution of the phase and spacing of 1,000,000 random matrices of order 10
sampled from the orthogonal group (top row), from the special orthogonal group
(middle row), and from the set of orthogonal matrices with determinant $-1$
(bottom row). In Figure~\ref{fig:dist-spacing-orthogonal-odd} we report the same
data for matrices of order 9, as the behavior of these
distributions changes dramatically depending on the parity of $n$.

The distribution of phase and spacing for the eigenvalues of matrices sampled
from $\On$ appears identical for both matrix dimensions we consider. In
particular, we note that in Figure~\ref{fig:eig-dist-On-even} and
Figure~\ref{fig:eig-dist-On-odd} there is a mass of probability corresponding to
the eigenvalues 1 and $-1$, which is a consequence of the fact that the
eigenvalues of real matrices always appear in conjugate pairs. Therefore, if $n$
is even then matrices with determinant $-1$ must always have both eigenvalues
$1$ and $-1$ (see Figure~\ref{fig:eig-dist-Ominus--even}), whereas if $n$ is odd
then all matrices with determinant 1 must have the eigenvalue 1 (see
Figure~\ref{fig:eig-dist-Oplus-odd}) and all those with determinant $-1$
must have the eigenvalue~$-1$ (see Figure~\ref{fig:eig-dist-Ominus-odd}).


\subsection{Timings and computational complexity}

\begin{figure}
  \centering
  \begin{tikzpicture}
    \begin{axis}[
      width=0.95\textwidth,
      height=8cm,
      grid=both,
      grid style={line width=.2pt, draw=gray!30},
      major grid style={line width=.2pt,draw=gray!60},
      ylabel={$t_{n}$},
      xlabel={$n$},
      ymode=log,
      xmode=log,
      log basis x={2},
      log basis y={10},
      ymin=1e-5,
      ymax=1e3,
      xmin=2,
      xmax=2^15,
      legend pos=north west
      ]
      \pgfplotstableread{figs/complexity.dat}\mydata;
      \addplot table [x index=0, y index=1] {\mydata};
      \addlegendentry{\newalg};
      \addplot table [x index=0, y index=2] {\mydata};
      \addlegendentry{\naivealg};
    \end{axis}
  \end{tikzpicture}
  \caption{Time $t_n$ (in seconds) required by \newalg{} and \naivealg{} to
    sample the eigenvalues of matrices of order $n$ between 2 and $2^{15}$. The
    tests for \naivealg\ have been performed only for $n$ up to
    $2^{13}$. \label{fig:complexity}}
\end{figure}
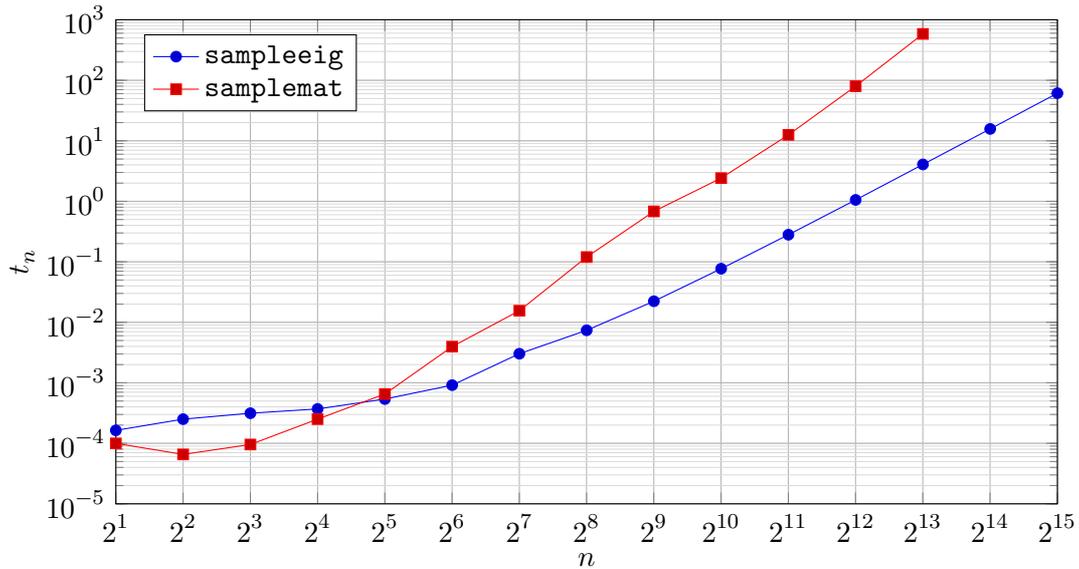

Now we compare the performance of our MATLAB implementations of
\newalg{} and \naivealg{}. Figure~\ref{fig:complexity} shows the time, in
seconds, required by the two algorithms to sample the eigenvalues of matrices of
order $n$ between 2 and $2^{15}$. For matrices of order up to $16$, \newalg{} is
slightly slower than \naivealg{}; this is due to the fact that normalizing the
rotations amounts to a large portion of the overall execution time of the
algorithm.

As the computational cost of this operation scales linearly, however, its
contribution becomes negligible as $n$ grows: for matrices of order $32$ and
above the execution time grows much faster for \naivealg{} than for \newalg{}.
This is expected, since the two algorithms have cubic and quadratic
computational cost, respectively.


\section{Conclusions}%
\label{sec:conclusions}

We have presented a method for sampling the joint distribution of the
eigenvalues of Haar-distributed orthogonal and unitary matrices. The two
ingredients of our approach are a technique for sampling the upper Hessenberg
form of Haar-distributed matrices, and an algorithm for computing the
eigenvalues of an $n \times n$ upper Hessenberg unitary or orthogonal matrix in
$\bigO(n^2)$ flops.

Our experimental results show that the new technique is more efficient than the
na{\"i}ve method that first samples a matrix from the Haar distribution and then
computes its eigenspectrum numerically. We used this algorithm to investigate
experimentally the distribution of the phase and spacing of the eigenvalues of
Haar-distributed matrices from $\Uplus$, $\On$, $\Oplus$, and $\Ominus$, groups
for which these distributions are not known explicitly.

\section*{Acknowledgments}

Preparation of the manuscript was carried out, in part, while the first author
was a Visiting Fellow at the University of Pisa. The authors thank Alan Edelman
for providing feedback on an early draft of the manuscript.


\bibliographystyle{siamplain-doi}
\bibliography{references.bib}

\end{document}